\documentclass{amsart}
\usepackage{graphicx}
\newtheorem{thm}{Theorem}[section]
\newtheorem{cor}[thm]{Corollary}

\theoremstyle{definition}

\theoremstyle{remark}

\numberwithin{equation}{section}

\begin{document}

\title[Hardy-Littlewood Maximal Operator And $BLO^{1/\log}$ Class of Exponents]{Hardy-Littlewood Maximal Operator And $BLO^{1/\log}$ Class of Exponents}%
\author{Tengiz Kopaliani}%
\author{Shalva Zviadadze}
\address[Tengiz Kopaliani]{Faculty of Exact and Natural Sciences, Javakhishvili Tbilisi State University, 13, University St., Tbilisi, 0143, Georgia}%
\address[Shalva Zviadadze]{Faculty of Exact and Natural Sciences, Javakhishvili Tbilisi State University, 13, University St., Tbilisi, 0143, Georgia}%
\email[Tengiz Kopaliani]{tengizkopaliani@gmail.com}%
\email[Shalva Zviadadze]{sh.zviadadze@gmail.com}

\thanks{The research of the first author is supported by Shota Rustaveli National Science Foundation grants no. 31/48 (Operators in some function spaces and their applications in Fourier Analysis) and  no. DI/9/5-100/13 (Function spaces, weighted inequalities for integral operators and problems of summability of Fourier series). Research of the second author is supported by Shota Rustaveli National Science Foundation grant \#52/36.}


\begin{abstract}
It is well known that if Hardy-Littlewood maximal operator is
bounded in space $L^{p(\cdot)}[0;1]$ then $1/p(\cdot)\in
BMO^{1/\log}$. On the other hand if $p(\cdot)\in BMO^{1/\log},$
($1<p_{-}\leq p_{+}<\infty$), then there exists $c>0$ such that
Hardy-Littlewood maximal operator is bounded in
$L^{p(\cdot)+c}[0;1].$ Also There  exists  exponent $p(\cdot)\in
BMO^{1/\log},$  ($1<p_{-}\leq p_{+}<\infty$) such that
Hardy-Littlewood maximal operator is not bounded in
$L^{p(\cdot)}[0;1]$.
\par In the present paper we construct exponent $p(\cdot),$ $(1<p_{-}\leq
p_{+}<\infty)$, $1/p(\cdot)\in BLO^{1/\log}$ such that Hardy-Littlewood maximal operator is not bounded in
$L^{p(\cdot)}[0;1]$.

\end{abstract}
\maketitle
\section{Introduction}

\par The variable exponent Lebesgue spaces $L^{p(\cdot)}(\mathbb{R}^{n})$
and the corresponding variable exponent Sobolev spaces
$W^{k,p(\cdot)}$ are of interest for their applications to the
problems in  fluid dynamics,  partial differential equations with
non-standard growth conditions, calculus of variations,  image
processing and etc (see \cite{DHHR}).

\par Given a measurable function
$p:[0;1]\rightarrow[1;+\infty),\,\,L^{p(\cdot)}[0;1]$ denotes
the set of measurable functions $f$ on $[0;1]$ such that for some
$\lambda>0$
$$
\int_{[0;1]}\left(\frac{|f(x)|}{\lambda}\right)^{p(x)}dx<\infty.
$$
This set becomes a Banach function spaces when equipped with the
norm
$$
\|f\|_{p(\cdot)}=\inf\left\{\lambda>0:\,\,
\int_{[0;1]}\left(\frac{|f(x)|}{\lambda}\right)^{p(x)}dx\leq1\right\}.
$$
\par For the given $p(\cdot),$ the conjugate exponent $p'(\cdot)$ is defined pointwise $p'(x)=p(x)/
(p(x)-1)$, $x\in[0;1].$ Given a set $Q\subset[0;1]$ we define some standard notations:
$$
p_-(Q):=\mathop{\mbox{essinf}}\limits_{x\in Q}p(x),\:\:\:\:
p_+(Q):=\mathop{\mbox{esssup}}\limits_{x\in Q}
p(x),\:\:\:\:p_-:=p_-([0;1]),\:\:\:\: p_+:=p_+([0;1]).
$$
\par Recall that the Hardy-Littlewood maximal operator is defined for
any $f\in L^{1}[0;1]$ by
$$
Mf(x)=\sup\limits_{x\in Q}\frac{1}{|Q|}\int_{Q}|f(t)|dt,
$$
where supremum is taken over all $Q\subset [0;1]$ intervals
containing point $x$ (assume  that sets like $[0;a)$ and $(a;1]$ are also intervals).  Denote  by $\mathcal{B}$ the class of all measurable exponents $p(\cdot),\,1< p_{-}\leq p_{+}<\infty$ for which the Hardy-Littlewood maximal operator is bounded on the space $L^{p(\cdot)}[0;1]$.
Different aspects concerning this class can be found in monographs \cite{CUF} and \cite{DHHR}.
\par Assume that $1<p_{-}\leq p_{+}<\infty$. The most important  condition, one widely used
in the study of variable Lebesgue spaces, is log-L\"{o}lder continuity. Let $C^{1/\log}$ denotes the set of exponents $p:[0;1]\rightarrow[1,+\infty)$ with log-H\"{o}lder condition 
$$
|(p(x)-p(y))\ln|x-y||\leq C,\,\, x,\,y\in[0;1],\,x\neq y.
$$

\par Diening \cite{DI} proved  a key consequence of log-H\"{o}lder continuity of $p(\cdot)$: if $1<p_-$ and
$p(\cdot)\in C^{1/\log},$ then $p(\cdot)\in \mathcal{B}.$ However, in some sense log-H\"{o}lder continuity is optimal. If the local log-H¨older continuity is replaced by a weaker uniform modulus
of continuity then this new condition is not sufficient in sense that there exists a variable exponent with this modulus of continuity such that $M$ is not bounded on $L^{p(\cdot)}[0;1]$ (see
\cite{CUF}, \cite{DHHR}).
\par The following characterization of exponents in $\mathcal{B}$ was
given in \cite{K2}.
\begin{thm}
\label{theorem_chi_asymp}
 $p(\cdot)\in\mathcal{B}$ if and only if
\begin{equation}
\label{uniform}
 \|\chi_{Q}\|_{p(\cdot)}\asymp
|Q|^{\frac{1}{|Q|}\int_{Q}\frac{1}{p(x)}dx} \,\,\,\,\mbox{and}
 \,\,\,\|\chi_{Q}\|_{p'(\cdot)}\asymp
|Q|^{\frac{1}{|Q|}\int_{Q}\frac{1}{p'(x)}dx}
\end{equation}
uniformly for all intervals $Q\subset[0;1].$
\end{thm}
\par Given a function $f\in L^{1}[0;1]$. Let define its $BMO$ modulus by
$$
\gamma(f,r)=\sup\limits_{|Q|\leq r}\frac{1}{|Q|}\int_{Q}|f(x)-f_{Q}|dx,\,\,0<r\leq1,
$$
where $f_Q$ denotes average of $f$ on the interval $Q$ and supremum is taken
over all intervals $Q\subset[0;1].$ We say that $f\in BMO^{1/\log}$ if $\gamma(f,r)\leq C/\log(e+1/r).$
\par The class $BMO^{1/\log}$ is very important for investigation of
exponents from $\mathcal{B}.$  Lerner \cite{Le} showed that if $a>0$ is small enough then exponent
\begin{equation}
\label{exponent_lerner_1} p(x)=2-a(1+\sin(\log\log(e+x+1/x))),
\end{equation}
belongs to class $BMO^{1/\log}\cap\mathcal{B}.$
\begin{thm} [\cite{Le}]
\label{theorem_BMO_VMO_c} Let $1<p_{-}\leq p_{+}<\infty$. If
$p(\cdot)\in BMO^{1/\log}$, then there exists $c>0$ such that
$p(\cdot)+c\in \mathcal{B}$.
\end{thm}
\par Taking into account Theorem \ref{theorem_BMO_VMO_c} the following questions are naturally
interesting:
\par\textbf{Question 1.} Let $p(\cdot)\in \mathcal{B}\cap BMO^{1/\log}.$ Does this imply
$p(\cdot)-\alpha\in \mathcal{B}$ for any $\alpha<p_{-}-1$?
\par\textbf{Question 2.} Let $p(\cdot)\in \mathcal{B}\cap BMO^{1/\log}.$ Does this imply
$\alpha p(\cdot)\in \mathcal{B}$ for any $\alpha>1/p_{-}$?

\par It has been shown by Lerner \cite{Ler} that all of these
questions have a negative answer. Slightly modified Lerner's
examples can be found in monograph \cite{DHHR}.

\par Let $\theta(t):=\min\{\max\{0,t+1/2\},1\},$ so that $\theta$ is
Lipschitz function with constant $1$. Define on $(0;1)$ the function
\begin{equation}
\label{exponent_lerner_diening} p(x)=2+6\theta(\sin(\pi\log\log(1/x))).
\end{equation}
Using exponent $p(\cdot),$ we may construct  examples of exponents which yields counterexamples for questions 1 and 2 (see \cite{DHHR}, pp. 148).

\par Given a function $f\in L^{1}[0;1]$. Let define its $BLO$ modulus by
$$
\eta(f,r)=\sup\limits_{|Q|\leq
r}(f_{Q}-\mathop{\mbox{essinf}}\limits_{x\in Q}f(x)),\,\,0<r\leq1,
$$
where supremum is taken over all intervals $Q\subset[0;1].$ We say that $
f\in BLO^{1/\log}$ if $\eta(f,r)\leq C/\log(e+1/r)$.

\par Note that the function
$$
f(x)=\left\{\begin{array}{rcl}
\ln\ln(1/x) & \textrm{if} & x\in (0,e^{-1}];\\
0 & \textrm{if} & x\in(e^{-1},1],
\end{array}
\right.
$$
belongs to $BLO^{1/\log}$ (see \cite{KSZ}).  The function $f$ is a classical example of the function from $BMO^{1/\log}$ (see \cite{Sp}). From the well-known observation that a Lipschitz function preserves mean oscillations, it follows that the functions (\ref{exponent_lerner_1}) and (\ref{exponent_lerner_diening}) provide an examples of a discontinuous bounded functions from $BMO^{1/\log}$.

\par It is natural to ask whether Lipschitz function preserves bounded lower oscillation? In general the answer of this question is negative. For exponent (\ref{exponent_lerner_diening}) we have $p(x)=2+6\theta(\sin(\pi f(x))),$ where $\theta$ is  Lipschitz function with constant $1$ and  $f\in BLO^{1/\log}.$  Note that exponent $p(\cdot)$  does not belong to $BLO^{1/\log}$.

\par Indeed, for $k\in \mathbb{N}$  define $a_k=e^{-e^{2k+11/6}},\:$ $\:b_k=e^{-e^{2k+7/6}}\:$, $\:c_k=e^{-e^{2k+5/6}}$, $\:d_k=e^{-e^{2k+1/6}}$.  Note that $p(x)=2$ on $(a_k;\, b_k)$ and $p(x)=8$ on $(c_k;\,d_k).$ It is clear that $a_k<b_k<c_k<d_k.$ Let $Q_k=(0;\,d_k).$  We have
$$
\frac{1}{|Q_k|}\int_{Q_k}p(x)dx-p_-(Q_k)=\frac{1}{d_k}\int_0^{d_k}(p(x)-2)dx\geq
$$
$$
\geq\frac{1}{d_k}\int_{c_k}^{d_k}(p(x)-2)dx=
\frac{6(d_k-c_k)}{d_k}\to6,\:\:\:\:k\to+\infty.
$$
Therefore $p(\cdot)\notin BLO^{1/\log}$.  Using the same arguments we can prove that $1/p(\cdot)\notin BLO^{1/\log}$. We only need assume that $Q_k=(0;b_k)$ (analogously we can prove that same statement is true for exponent defined by (\ref{exponent_lerner_1})).

\par It is clear that $BLO^{1/\log}\subset BMO^{1/\log},$ therefore by theorem \ref{theorem_BMO_VMO_c}
for any exponent $p(\cdot),\,\,1/p(\cdot)\in BLO^{1/\log},$
$1<p_-\leq p_+<+\infty$ there exits constant $c>0$ such that $p(\cdot)+c\in\mathcal{B}$.

\par Taking into account the last statement the analogous of questions 1 and 2 are interesting:

\par  \textbf{Question 3.} Let for exponent $p(\cdot),\,1<p_{-}\leq p_{+}<\infty$ and
 for $c>0$ we have  $1/(p(\cdot)+c)\in BLO^{1/\log}$ and $p(\cdot)+c\in\mathcal{B}$.
Does this implies $p(\cdot)\in\mathcal{B}$?
\par \textbf{Question 4.} Let for exponent $p(\cdot),\,1<p_{-}\leq p_{+}<\infty$ and
for $c>1$ we have  $1/(cp(\cdot))\in BLO^{1/\log}$ and $cp(\cdot)\in\mathcal{B}$. Does
this implies $p(\cdot)\in\mathcal{B}$?
From Theorem \ref{theorem_chi_asymp} and questions 3 and 4, it follows a new question:
\par\textbf{Question 5.} Let $1<p_{-}\leq p_{+}<\infty$ and uniformly
for all intervals $Q\subset[0;1]$ we have one of the following asymptotic estimations
\begin{equation}
\|\chi_{Q}\|_{p(\cdot)}\asymp
|Q|^{\frac{1}{|Q|}\int_{Q}\frac{1}{p(x)}dx}\,\,\,\,\:\:\:\mbox{or}\:\:\:\:\,\,\,\|\chi_{Q}\|_{p'(\cdot)}\asymp |Q|^{\frac{1}{|Q|}\int_{Q}\frac{1}{p'(x)}dx}.
\end{equation}
Does this chosen estimation implies the second one asymptotic estimation?
\par We obtain a negative answer of this question.
\begin{thm}
\label{theorem_kontr_makenhaupt}
There exists exponent $p(\cdot)$ such that  $1/p(\cdot)\in
BLO^{1/\log},\,\,1<p_{-}\leq p_{+}<\infty$ and we have only one
uniformly asymptotic estimate \textnormal{(\ref{uniform})}.
\end{thm}

An immediate consequence of Theorem \ref{theorem_kontr_makenhaupt} and Theorem \ref{theorem_chi_asymp} is following corollary.  

\begin{cor}
\label{corollary_BLO_c_neq_0} There exists exponent $p(\cdot)$ such
that  $1/p(\cdot)\in BLO^{1/\log},\,\,1<p_{-}\leq p_{+}<\infty$  but
$p(\cdot)\notin \mathcal{B}$.
\end{cor}

Using the same arguments as in \cite{DHHR} we obtain  negative answer of above questions 3 and 4.

\par In this paper we shall also consider the Hardy operator defined by
$$
Tf(x)=\frac{1}{x}\int_0^xf(t)dt.
$$
where $f\in L^{p(\cdot)}[0;1].$
\par We will prove following theorem.

\begin{thm}
\label{theorem_on_hady_oper} There exists exponent
$p(\cdot),\,1<p_{-}\leq p_{+}<\infty ,$ such that $1/p(\cdot)\in
BLO^{1/\log}$ but Hurdy's operator $T$ is not bounded in
$L^{p(\cdot)}[0;1]$.
\end{thm}

\section{Proof of results}

\begin{proof}[Proof of Theorem \ref{theorem_kontr_makenhaupt}]
\par Let  $d_n=e^{-e^n},\,\,n\in\{0\}\cup\mathbb{N}$ and
$c_0=2/e$, $c_{2n+1}=c_{2n}-(d_n-d_{n+1})$, $c_{2n+2}=c_{2n}-2(d_n-d_{n+1})$,  $n\in\{0\}\cup\mathbb{N}$. Let
$$
g(x)=\left\{\begin{array}{rcl}
\ln\ln\frac{1}{c_{2n}+c_{2n+2}-x-d_n}-n & \textrm{if} & x\in (c_{2n+2};\,c_{2n+1}], n\in\{0\}\cup\mathbb{N}\vspace{1ex};\\
\ln\ln\frac{1}{x-d_n}-n & \textrm{if} & x\in(c_{2n+1};\,c_{2n}], n\in\{0\}\cup\mathbb{N}\vspace{1ex};\\
0 & \textrm{if} & x\in(2/e,1].
\end{array}
\right.
$$
Note that function $g$ is  bounded  function (some sense analogous
of $\sin(f(x))$) which belongs to $BLO^{1/\log}$ (see \cite{KSZ}).

Let denote
$$
A:=\{x:\,\,g(x)<1/2\}=\cup_{k=0}^{\infty}(\alpha_k;\,\beta_{k})
$$
and
$$
B:=\{x:\,\,g(x)>63/64\}=\cup_{k=0}^{\infty}(a_k;\,b_{k}),
$$
where the numbers $a_k,\,b_k,\,\alpha_k,\,\beta_k$ are defined as $\alpha_0=e^{-1}+e^{-e^{1/2}}$, $\:\:\beta_0=1,\:\:$ $\alpha_k=e^{-e^{k}}+e^{-e^{k+1/2}},\:\:$ $\beta_k=e^{-e^{k-1}}+2e^{-e^{k}}-e^{-e^{k-1/2}},\:\:$ $k\in\mathbb{N}$ and
$a_k=e^{-e^{k}}+2e^{-e^{k+1}}-e^{-e^{k+63/64}},\:\:$ $b_k=e^{-e^{k}}+e^{-e^{k+63/64}},\:\;$ $k\in\{0\}\cup\mathbb{N}.
$

It is clear that,
$\beta_{k+1}<a_k<c_{2k+1}<b_k<\alpha_k<c_{2k}<\beta_k$ for
$k\in\{0\}\cup\mathbb{N}.$ We have
\begin{equation}
\label{estim_beta_k_alpha_k}
\frac{\beta_{k}}{\alpha_{k}}>\frac{\beta_{k}}{c_{2k}}=\frac{2e^{-e^{k}}+
e^{-e^{k-1}}-e^{-e^{k-1/2}}}{2e^{-e^{k}}}=1+e^{e^{k}(1-e^{-1})}/2-e^{e^{k}(1-e^{-1/2})}/2,
\end{equation}
and consequently $\alpha_{k}/\beta_{k}\rightarrow 0$, when
$k\rightarrow\infty.$
\par We also have
\begin{equation}
\label{estim_b_k-a_k_beta_k}
\frac{b_k-a_k}{\beta_k^{16}}=\frac{e^{-e^{k}}+e^{-e^{k+63/64}}-e^{-e^{k}}-2e^{-e^{k+1}}+e^{-e^{k+63/64}}}{\beta_k^{16}}=
\end{equation}
$$
=\frac{2e^{-e^{k+63/64}}-2e^{-e^{k+1}}}{\beta_k^{16}}\geq\frac{2e^{-e^{k+63/64}}-2e^{-e^{k+1}}}{\left(4e^{-e^{k-1}}\right)^{16}}
\simeq\frac{2e^{-e^{k+63/64}}}{4^{16}\left(e^{-e^{k-1}}\right)^{16}}\geq
$$
$$
\geq e^{-e^{k+1}}\cdot4^{-16}\left(e^{-e^{k-1}}\right)^{-2e^2}=e^{e^{k+1}}4^{-16}\to+\infty,\:\:\:
k\to+\infty.
$$

\par Let now construct exponent $p(\cdot)$ in following way

\begin{equation}
\label{exponent_ksz}
p(x)=\left\{\begin{array}{rcl}
2 & \textrm{if} & x\in A \vspace{1ex};\\
64/63 & \textrm{if} & x\in B\vspace{1ex};\\
1/g(x) & \textrm{if} &
x\in[0;1]\backslash(A\cup B).\\
\end{array}
\right.
\end{equation}

 Since $g\in BLO^{1/\log}$ (see \cite{KSZ}), then $1/p(\cdot)\in BLO^{1/\log}$. Consequently we have uniformly asymptotic estimation (\ref{uniform}) for  norms $\|\chi_{Q}\|_{p(\cdot)}$ (see
 \cite{KSZ}).

\par We will prove that asymptotic estimation
\begin{equation}
\label{asimtotika}
||\chi_Q||_{p'(\cdot)}\asymp|Q|^{\frac{1}{|Q|}\int_Q\frac{1}{p'(x)}dx}
\end{equation}
is not valid.
 It is clear that
$$
p'(x)=\left\{\begin{array}{rcl}
2 & \textrm{if} & x\in A \vspace{1ex};\\
64 & \textrm{if} & x\in B\vspace{1ex};\\
p(x)/(p(x)-1) & \textrm{if} &
x\in[0;1]\backslash(A\cup B).\\
\end{array}
\right.
$$
\par Let $Q_k=(0;\beta_k).$ By (\ref{estim_beta_k_alpha_k}) we have
$$
\frac{1}{|Q_k|}\int_{Q_k}\frac{1}{p'(x)}dx\geq
\frac{1}{\beta_k}\int_{\alpha_k}^{\beta_k}\frac{1}{p'(x)}dx=
\frac{\beta_k-\alpha_k}{2\beta_k}\rightarrow\frac{1}{2},\:\:\:k\to+\infty.
$$
Therefore there exists $k_0$ such that
$$
\frac{1}{|Q_k|}\int_{Q_k}\frac{1}{p'(x)}dx\geq\frac{1}{4}, \qquad k\geq k_0.
$$
Now consider $\lambda>1.$ For  $k\geq k_0$ by (\ref{estim_b_k-a_k_beta_k}) we have
$$
\int_{Q_k}\left(\lambda |Q_k|^{\frac{1}{|Q_k|}\int_{Q_k}\frac{1}{p'(t)}dt}\right)^{-p'(x)}dx\geq\int_{Q_k}
\left(\lambda |Q_k|^{\frac{1}{4}}\right)^{-p'(x)}dx\geq
$$
$$
\geq\int_{a_k}^{b_k}\left(\lambda |Q_k|^{\frac{1}{4}}\right)^{-p'(x)}dx=\frac{b_k-a_k}{\lambda^{64}\beta_k^{16}}\rightarrow\infty,\,\,k\rightarrow\infty.
$$
As a consequence, we obtain that asymptotic estimation (\ref{asimtotika}) is not valid.
\end{proof}

\begin{proof}[Proof of Corollary \ref{corollary_BLO_c_neq_0}]
Consider exponent $p(\cdot)$ constructed in Theorem \ref{theorem_kontr_makenhaupt} then by Theorem \ref{theorem_chi_asymp} we get desired result.
\end{proof}

\begin{proof}[Proof of Theorem \ref{theorem_on_hady_oper}]
Consider exponent $p(\cdot)$ constructed in Theorem \ref{theorem_kontr_makenhaupt}.  We will
prove that Hardy's operator is not bounded in $L^{p(\cdot)}[0;1]$.
 Note that the condition
$$
\sup\limits_{0<s\leq1}||\chi_{[0;s]}(x)||_{p'(\cdot)}\cdot\left\|x^{-1}\chi_{[s;1]}(x)\right\|_{p(\cdot)}<+\infty
$$
is necessary for boundedness of Hardy's operator in
$L^{p(\cdot)}[0;1]$ (see \cite{K1}).  Let us check that in our case  this condition fails. 
\par We have
$$
\int_0^{\beta_k}\frac{1}{\lambda^{p'(x)}}dx\geq \int_{a_k}^{b_k}\frac{1}{\lambda^{p'(x)}}dx=
\int_{a_k}^{b_k}\frac{1}{\lambda^{64}}dx=\frac{b_k-a_k}{\lambda^{64}}.
$$
From the last estimation and definition of the norm in variable exponent Lebesgue space we conclude that
$$
||\chi_{(0;\beta_k]}||_{p'(\cdot)}\geq\sqrt[64]{b_k-a_k}.
$$

We also have
$$
\int_{\beta_k}^1\left(\frac{1}{x\lambda}\right)^{p(x)}dx\geq\int_{\alpha_{k-1}}^{\beta_{k-1}}\left(\frac{1}{x\lambda}\right)^{p(x)}dx=
\int_{\alpha_{k-1}}^{\beta_{k-1}}\left(\frac{1}{x\lambda}\right)^2dx=
$$
$$
=\frac{1}{\lambda^2}\int_{\alpha_{k-1}}^{\beta_{k-1}}\frac{1}{x^2}dx=\frac{1}{\lambda^2}\left(\frac{1}{\alpha_{k-1}}-\frac{1}{\beta_{k-1}}\right).
$$
By the last estimation we conclude that
$$
\left\|x^{-1}\chi_{[\beta_k;1)}(x)\right\|_{p(\cdot)}\geq
\sqrt{1/\alpha_{k-1}-1/\beta_{k-1}}.
$$
Finally we get
$$
||\chi_{(0;\beta_k]}(x)||_{p'(\cdot)}\cdot\left\|x^{-1}\chi_{[\beta_k;1)}(x)\right\|_{p(\cdot)}\geq\sqrt[64]{b_k-a_k}\cdot\sqrt{1/\alpha_{k-1}-1/\beta_{k-1}}\geq
$$
$$
\geq\sqrt[64]{\frac{2}{e^{e^{k+63/64}}}}\sqrt[64]{1-\frac{e^{e^{k+63/64}}}{e^{e^{k+1}}}}\cdot\sqrt{\frac{1}{\alpha_{k-1}}}\cdot\sqrt{1-
\frac{\alpha_{k-1}}{\beta_{k-1}}}\simeq
$$
$$
\simeq\sqrt[64]{\frac{2\left(e^{e^{k-1}}\right)^{32}}{e^{e^{k+63/64}}}}\geq\sqrt[64]{\frac{\left(e^{e^{k-1}}\right)^{3e^2}}{e^{e^{k+1}}}}=\sqrt[32]{e^{e^{k+1}}}\to+\infty,\qquad k\to+\infty.
$$
\end{proof}


\end{document}